\documentclass[11pt]{amsart}

\usepackage{amssymb,amscd,amsthm,amsxtra}
\usepackage{latexsym}

\newtheorem{thm}{Theorem}[section]

\theoremstyle{definition}
\newcommand{\comment}[1]{}

\theoremstyle{remark}

\numberwithin{equation}{section}
\newcommand{\R}{\mathbb R}

\begin{document}

\title[A priori estimates for NLS]{Remarks on global a priori estimates for the nonlinear Schr\"odinger equation.}
\author{J. Colliander}
\address{Department of Mathematics, University of Toronto, Toronto, ON, Canada M5S 2E4}
\email{\tt colliand@math.toronto.edu}

\author{M. Grillakis}
\address{Department of Mathematics, University of Maryland,
College Park, MD 20742} \email{\tt mng@math.umd.edu}

\author{N. Tzirakis}
\address{Department of Mathematics, University of Illinois at Urbana-Champaign, Urbana, IL, 61801, USA}
\email{tzirakis@math.uiuc.edu}
\date{7 July 2009}


\begin{abstract}
We present a unified approach for obtaining global a priori estimates for solutions of nonlinear defocusing Schr\"odinger equations with defocusing nonlinearities. The estimates are produced by contracting
 the local momentum conservation law with appropriate vector fields. The corresponding law is written for defocusing equations of tensored solutions. 
 In particular, we obtain a new estimate in two dimensions. We bound
 the restricted $L_t^4L_{\gamma}^4$ Strichartz norm of the solution on
 any curve $\gamma$ in $\mathbb R^2$. For the specific case of a straight line we upgrade
 this estimate to a weighted Strichartz estimate valid in the full plane.   
\end{abstract}
\maketitle

\section{Introduction}
In this paper, we prove global-in-time mixed Lebesgue spaces estimates for solutions of the semilinear Schr\"odinger equation
\begin{equation}\label{nls}
\left\{
\begin{matrix}
iu_{t}-\Delta u +|u|^{p-1}u=0, & x \in {\mathbb R^n}, & t\in {\mathbb R},\\
u(x,0)=u_{0}(x)\in H^{s}({\mathbb R^n})
\end{matrix}
\right.
\end{equation}
for any $p>1$. 
Equation \eqref{nls} satisfies the following conservation laws:
Energy conservation
\begin{equation}\label{energy}
E(u)(t)=\frac{1}{2}\int |\nabla u(t)|^{2}dx+\frac{1}{p+1}\int |u(t)|^{p+1}dx=E(u_{0}). 
\end{equation}
Mass conservation 
\begin{equation}
\|u(t)\|_{L^{2}}=\|u_{0}\|_{L^{2}},\label{mass}
\end{equation}
and momentum conservation
\begin{equation}\label{momentum}
\vec{p}(t)=\Im \int_{\Bbb R^n}\bar{u}\nabla u dx.
\end{equation}
The conservation laws of mass and energy identify $H^1$ as the energy space. Moreover, membership of the solution in this space provides an a priori bound on the $H^1$ norm of the solution. Since the problem is
locally well-posed in $H^{1}$ whenever $1<p<1+\frac{4}{n-2}$ for $n
\geq 3$ and for $1<p<\infty$ for $n=1,2$, the conservation control enables one to iterate the local-in-time solutions over any large time
 interval maintaing the same uniform $H^{1}$ bound. Thus, the problem is globally well-posed. For all the relevant definitions and the local properties of the solutions the reader can consult 
\cite{tc}, \cite{tt}.  
Since the Cauchy problem is globally well-posed, our attention turns
toward the issue of describing and classifying the asymptotic behavior in time
 for global solutions. A possible method to attack this issue is to compare the given dynamics with suitably chosen simpler asymptotic dynamics.
 The method applies to a wide variety of dynamical systems and in
 particular to some systems defined by nonlinear PDE, and give rise to
 the scattering theory. 
 
For the semilinear problem \eqref{nls}, the first obvious candidate
for the simplified asymptotic behavior is the free dynamics generated by the group $S(t)=e^{-it\Delta}$. The comparison
 between the two dynamics gives rise to the questions of the existence of wave operators and of the asymptotic completeness of the solutions. More precisely we have:
\\
\\
{\bf (a)} Let $v_{+}(t)=S(t)u_{+}$ be the solution of the free equation. Does there exist a solution $u$ of equation \eqref{nls} which behaves asymptotically
 as $v_{+}$ as $t \rightarrow \infty$, typically in the sense that , $\|u(t)-v_{+}\|_{H^{1}} \rightarrow 0, \ \ \mbox{as $t \rightarrow \infty$.}$ If this is 
true then one can define the map $\Omega_{+}: u_{+} \rightarrow u(0)$. The map is called the wave operator and
 the problem of existence of $u$ for given $u_{+}$ is referred to as the problem of the {\it{existence of the wave operator}}. The analogous problem
 arises as $t \rightarrow -\infty$.
\\
\\ 
{\bf (b)} Conversely, given a solution $u$ of \eqref{nls}, does there exist an asymptotic state $u_{+}$ such that $v_{+}(t)=S(t)u_{+}$ behaves 
asymptotically as $u(t)$, in the above sense. If that is the case for any $u$ with initial data in $X$ for some $u_{+} \in X$,
 one says that {\it{asymptotic completeness}} holds in $X$.
\\
\\ 
Asymptotic completeness is a much harder problem than the existence of
the wave operators except in the case of small data theory which
follows from the iteration method proof of the local
 well-posedness. Asymptotic completeness for large data requires a repulsive nonlinearity and
 usually proceeds through the derivation of a priori estimates for general solutions. The question of scattering or in general the question of dispersion of the nonlinear solution is tied to weather there is
 some sort of decay in a certain norm, such as the $L^p$ norm for $p>2$. In particular knowing the exact rate of decay of various $L^p$ norms for the linear solutions, it would be ideal to obtain estimates
 that establish similar rates of decay for the nonlinear problem. For
 a discussion comparing the decaying rates of the linear and the nonlinear problem the reader can consult \cite{tc} and the
 references therein. The decay of the linear solutions can immediately
 establish weak quantum scattering in the energy space but to estimate
 the linear and the nonlinear dynamics in the energy norm we usually
 look to prove that the $L^p$ norm of the nonlinear solution to go to zero as $t \rightarrow \infty$.  

Strichartz type estimates assure us that certain $L^p$ norms decay
asymptotically for large time but only for the linear part of the solution. For the nonlinear part we need to obtain 
general decay estimates on solutions of defocusing equations. The mass and energy conservation laws establish the boundedness of the $L^2$ and the $H^1$ norms but are 
insufficient to provide a decay for higher powers of Lebesgue norms. In this note we provide a summary of recent and new results that 
demonstrates a straightforward method to obtain such estimates by taking advantage of the momentum conservation law \eqref{momentum}. One can 
construct various quantities based on the momentum density which are monotone in time and so the fundamental theorem of calculus provides space-time bounds for general solutions.

Conservation laws come in both integral
 and differential forms. The differential form of the conservation law
 includes flux terms and can sometimes be localized in space-time by integrating against suitable cut-off functions or
 contracting against a suitable vector field to produce an almost
 conserved or monotone quantity. Thus, from a single differential conservation 
law, one can generate a variety of useful estimates, which can constrain the direction of propagation of a solution or provide a decay estimate for the solution or components of the solution (e.g. the low or high frequency part of the solution).

A key example of these ideas is contained in the following generalized virial inequality{\footnote{In fact, one can write an identity.}}
 of Lin and Strauss \cite{LinStrauss}. (We recall the proof of this inequality in the next section.)
\vskip 0.1 in
\begin{equation}\label{linstr}
\\
\int_{0}^{T}\int_{\Bbb R^n}(-\Delta \Delta a(x))|u(x,t)|^2dxdt+\int_{0}^{T}\int_{\Bbb R^n}(\Delta a(x)) {|u(x,t)|^{p+1}} dxdt \lesssim
 \sup_{[0,T]}|M_{a}(t)|
\end{equation}
\\
where $a(x)$ is a convex function, $u$ is a solution to \eqref{nls} and $M_{a}(t)$ is the {\it{Morawetz action}} defined by
\\
\begin{equation}\label{mora}
M_{a}(t)=2\int_{\Bbb R^n}\nabla a(x) \cdot \Im(\bar{u}(x)\nabla u(x)) dx.
\end{equation}
\vskip 0.1 in
Note that throughout the text we use the symbol $\lesssim$ to supress inessential constants from the inequalities we present. An inequality of this form was first derived in the context of the
Klein-Gordon equation by Morawetz \cite{cm} and then extended to the NLS equation in \cite{LinStrauss}. The inequality was applied to prove asymptotic completeness
 first for the nonlinear Klein-Gordon and then for the NLS equation in the papers by Morawetz and Strauss, \cite{ms}, and by Lin and Strauss, \cite{LinStrauss} for slightly more regular solutions in space dimension
$n \geq 3$. The case of general finite energy solutions for $n \geq 3$ was treated in \cite{gv} for the NLS and in \cite{gv1} for the Hartree. The treatment was then improved to the harder case of low dimensions by
 Nakanishi, \cite{kn}, \cite{kn1}. The vector field method that we outline in this paper, applied to tensor product of two or more different solutions, gives stronger estimates and simplifies the proofs of
 the results in the papers cited above. The method of contracting the point-wise conservation laws with a vector field is not new. It has been used extensively for the wave equation in the past. See \cite{tt} and
 the references therein for more information.

The weight function that was commonly used in the past was $a(x)=|x|$. 
This choice has the advantage that the distribution $-\Delta \Delta (\frac{1}{|x|})$ is positive for $n \geq 3$. More precisely it is easy to compute that $\Delta a(x)=\frac{n-1}{|x|}$ and that
$$-\Delta \Delta a(x)=\left\{
\begin{array}{cc}
8\pi\delta(x), & \text{if} \quad  n=3\\
\frac{(n-1)(n-3)}{|x|^2}, & \text{if} \quad n \geq 4.
\end{array}
\right.
$$
In particular, the computation gives the following estimate for $n=3$
\begin{equation}\label{linstr}
\int_{\Bbb R_t}|u(t,0)|^2dt+\int_{0}^{T}\int_{\Bbb R^3}\frac{|u(x,t)|^{p+1}}{|x|} dxdt \lesssim \sup_{t}\|u(t)\|_{\dot{H}^{\frac{1}{2}}}^{2}.
\end{equation}
The second, nonlinear term, or certain local versions of it, have played central role in the scattering theory for the 
 nonlinear Schr\"odinger equation, \cite{jbou}, \cite{gv}, \cite{mg}, \cite{LinStrauss}. The first term did not play a big role in
 these works.
The fact that in 3d the bi-harmonic operator acting on the weight
$a(x)$ produces the $\delta-$measure can be exploited further. The
following idea, first appearing in 
  \cite{ckstt4}, takes advantage of the first linear term and provides
  a global a priori estimate valid for any defocusing nonlinearity.
We consider two separate solutions of the nonlinear Schr\"odinger equation and take the tensor product of them in $\Bbb R^{6}:= \Bbb R^3 \times \Bbb R^3$. We define a new weight $a(x)=|y-z|$ where 
$x=(y,z) \in \Bbb R^3 \times \Bbb R^3$. We have provided the details of the calculation in the next section but let us mention that the heart of the matter is that
$$-\Delta_x \Delta_x (|y-z|)=32\pi \delta(y-z)$$ 
where $\Delta_x=\Delta_{y}+\Delta_{z}$. We can now define a new
Morawetz action that exploits the correlation of the two mass
densities and the positivity of its time derivative. Substituting this result in the first linear term using the new density 
\\
$$|u(x,t)|^{2}=|u(y,t)|^2|u(z,t)|^2,$$
we obtain 
$$\int_{\Bbb R_t}\int_{\Bbb R^3}|u(x,t)|^{4}dx \lesssim C$$
\\
for solutions that stay in the energy space.
This estimate is not immediately useful for extending the energy-critical theory and removing the radial assumption on the solutions that was achieved in \cite{jbou}. 
Nevertheless a frequency localized version of this estimate has been 
successfully implemented to remove the radial assumption, and prove global well-posedness and scattering for the energy-critical (quintic) equation in 3d, \cite{ckstt5}. Further applications include a
very simple proof of the 3d energy scattering result of Ginibre and Velo, \cite{gv} for the $L^2$ super-critical NLS, which is the case for any $p>\frac{7}{3}$. Similar estimates can be obtained in any dimension
 higher than three and they can treat the higher dimensional scattering theory in a unified way. For details see, \cite{tvz1}. The idea to bypass the problem of the quite nasty
 distribution $-\Delta \Delta (|x|)$ in low dimensions and prove similar estimates without a restriction on the dimension was taken in \cite{cgt}. See also \cite{pv}.

In this note we want to comment on the original idea of using the 3d reduction on the solutions and obtain linear Strichartz type estimates. Again the idea is to take advantage of the
special calculation in 3d of the weight function $a(x)$. The numerology is as follows: We correlate two 3d solutions and we define a solution in $6=3+3$ dimensions. If we integrate the product density against the 
$\delta$ distribution  we obtain an estimate for solutions in 3d, and thus we managed a reduction of $3=6-3$. 
It is not clear how to take advantage of this idea in other dimensions. To obtain for example a 1d estimate, using this idea and after a ``dimensional reduction'' of 3 we have to correlate four solutions. 
We can thus define a weight function in four dimensions. The appropriate weight function is the distance of a point in $\Bbb R^4$ from the diagonal $(x,x,x,x)$. This provides an $L_{t}^{8}L_{x}^{8}$ 
a priori estimate, \cite{chvz}. If the dimension is higher than 4, correlation of two solutions, produces a solution on $\Bbb R^n$, when $n$ is at least $8$ and the reduction to $5=8-3$ is not immediately useful.
 Of course correlation of more than two, apart form the technical issues, obviously doesn't improve the situation for numerical reasons. If we correlate two solutions in two dimensions we obtain a
 solution in 4d and after the reduction we are left with a possible 1d estimate. This is the approach we follow in this paper. We define the weight function to be the distance of a point in $\Bbb R^4$ from the
 diagonal $\vec{x}_1=\vec{x}_2=\vec{x}(l)$ where $\vec{y}=(\vec{x}_1,\vec{x}_2) \in \Bbb R^2 \times \Bbb R^2$ and $\vec{x}(l)$ is the position vector of a straight line. We obtain an estimate that involves a line
 integral that can be upgraded in a full 2d integral by an averaging argument. We present the calculations in section 2 but our theorem reads as follows:
\begin{thm}[Correlation estimate in two dimensions]
\label{thm1}
 Let $u$ be an $H^{\frac{1}{2}}$ solution to \eqref{nls} on the space-time slab $I\times {\Bbb R}^2$. Then for any $x_{0} \in \Bbb R^2$ we have
\begin{equation}\label{22d}
\sup_{x_{0}}\int_{I}\int_{\Bbb R^2} \frac{|u(x,t)|^{4}}{|x-x_{0}|}dxdt \lesssim  \sup_{t}\|u(t)\|_{L^2}^2\|u(t)\|_{\dot{H}^{\frac{1}{2}}}^{2}
\end{equation}
\end{thm}
Note that we have dropped the nonnegative term arising from the
nonlinear term from the right side of \eqref{22d}. Thus, the linear
estimate \eqref{22d} is valid for all nonlinear evolutions with a
defocusing nonlinear term.

\vskip 0.1 in
\section{Method and its applications.}
We start{\footnote{Similar arguments apply in the setting of a more
    general defocusing nonlinearity.}} with the equation 
\begin{equation}\label{nls2}
iu_{t}-\Delta u=|u|^{p-1}u
\end{equation}
with $p \geq 1$. The two dimensional estimate that we will derive is linear and thus holds for all values of $p \geq 1$. For that reason we restrict ourselves to the important physical case of the cubic 
and thus we pick $p=3$. The reader will notice that we use Einstein's summation convention throughtout the paper. According to this convention, when an index variable appears twice in a single term, once in an upper (superscript) and once in a lower (subscript) position, it implies that we are summing over all of its possible values.
\\
\\
We define the mass density $\rho$ and the momentum vector $\vec{p}$, by the relations
$$\rho=\frac{1}{2}|u|^2,\ \ \ \ \ \ \ \ \ \ \ \ p_{k}=\Im (\bar{u}\nabla_{k}u).$$
It is well known that the abstract solution to the semilinear Schr\"odinger equation satisfies mass and momentum conservation. The local conservation of mass reads
\begin{equation}\label{conmass}
\partial_{t}\rho=div{\vec{p}}=\nabla_{j}p^{j}
\end{equation}
and the local momentum conservation is 
\begin{equation}\label{conmom}
\partial_{t}p_{k}=\nabla_{j}\left( \delta_{k}^{j}(2\rho^{2}-\Delta \rho)+\sigma_{k}^{j} \right)
\end{equation}
where the tensor $\sigma_{kj}$ is given by
$$\sigma_{jk}=2\Re(\nabla_{j}u\nabla_{k}\bar{u})=\frac{1}{\rho}(p_{j}p_{k}+\nabla_{j}\rho \nabla_{k}\rho).$$
Notice that the term $2\rho^{2}$ is the only nonlinear term that
appears in the expression. For the general monomial defocusing
nonlinearity $|u|^{p-1}u$ this term would take the form $2^{\frac{p+1}{2}}\frac{p-1}{p+1}\rho^{\frac{p+1}{2}}$.
We now define the Morawetz action
$$M_{a}(t)=-\int_{\Bbb R^{n}}\nabla a \cdot \vec{p}\ dx=-\int_{\Bbb R^{n}}\nabla_{j}a \ p^{j}\ dx$$ 
where the weight function $a(x)$ is given by $a(x)=|x|$ and thus $\nabla_{j}a=\frac{x_{j}}{|x|}$.
The following computation is due to Lin and Strauss, \cite{LinStrauss}, although we have adapt the proof to serve our own purposes. If we contract the momentum conservation equation with the vector field 
$\vec{X}=\nabla a$ we
 obtain
$$X^{k}\ \partial_{t}p_{k}=X^{k}\ \nabla_{j}\left( \delta_{k}^{j}(2\rho^{2}-\Delta \rho)+\sigma_{k}^{j} \right)$$ or that
$$\partial_{t}(X^{k}\ p_{k})=\nabla_{j}\{X^{k}\ \delta_{k}^{j}(2\rho^{2}-\Delta \rho)+X^{k}\ \sigma_{k}^{j}\}-(\nabla_{j}X^{k})\delta_{k}^{j}(2\rho^2-\Delta \rho)-(\nabla_{j}X^{k})\sigma_{k}^{j}.$$
\\
Integrating over the whole space, assuming enough decay for the solutions (an assumption that can be easily removed by a standard approximation argument), we have
\\
$$-\partial_{t}\int_{\Bbb R^{n}}p_{k}X^{k}dx=\int_{\Bbb R^n}(\nabla_{j} X^{j})(2\rho^2-\Delta \rho)dx+\int_{\Bbb R^n}(\nabla_{j} X^{k})\sigma_{k}^{j}dx=$$
\\
$$\int_{\Bbb R^n}(div X)(2\rho^2-\Delta \rho)dx+\int_{\Bbb R^n}(\nabla_{j} X^{k})\sigma_{k}^{j}dx.$$
\\
Now we compute
$$\nabla_{j}X^{k}=\frac{\delta_{j}^{k}|x|^{2}-x_{j}x^{k}}{|x|^3},$$ and
$$div X=\frac{n-1}{|x|},$$
and thus
\\
$$\partial_{t}M=-\partial_{t}\int_{\Bbb R^{n}}p_{k}X^{k}\ dx=(n-1)\int_{\Bbb R^n}\frac{1}{|x|}(2\rho^2-\Delta \rho)\ dx+\int_{\Bbb R^n}(\nabla_{j} X^{k})\sigma_{k}^{j}\ dx=$$
\\
$$2(n-1)\int_{\Bbb R^n}\frac{\rho^2}{|x|}\ dx+(n-1)\int_{\Bbb R^{n}}\left(-\Delta (\frac{1}{|x|})\right)\rho\ dx+\int_{\Bbb R^n}(\nabla_{j} X^{k})\sigma_{k}^{j}\ dx$$
\\
by integration by parts. The first term of the right hand side is positive. In addition for $n \geq 3$ the distribution $-\Delta (\frac{1}{|x|})$ is positive and thus we have
\\
$$\partial_{t}M \geq 2(n-1)\int_{\Bbb R^n}\frac{\rho^2}{|x|}\ dx+\int_{\Bbb R^n}(\nabla_{j} X^{k})\sigma_{k}^{j}\ dx.$$
\\
If we recall that $\sigma_{jk}=2\Re(\nabla_{j}u\nabla_{k}\bar{u})$ an explicit calculation reveals that 
\\
$$(\nabla_{j} X^{k})\sigma_{k}^{j}=\frac{2}{|x|}\left( |\nabla u|^{2}-\frac{1}{|x|^{2}}|(x\cdot \nabla)u|^{2}\right) \geq 0,$$
\\
and thus for any $n\geq 3$ we have
$$\partial_{t}M \geq 2(n-1)\int_{\Bbb R^n}\frac{\rho^2}{|x|}\ dx.$$ 
\\
Applying the fundamental theorem of calculus we obtain
$$\int_{\Bbb R_{t}}\int_{\Bbb R^n}\frac{|u(x,t)|^{4}}{|x|}\ dx \lesssim \sup_{t} |M(t)|.$$
\\
By definition of $M(t)$ and Hardy's inequality, \cite{ckstt4}, we obtain that
$$\int_{\Bbb R_{t}}\int_{\Bbb R^n}\frac{|u(x,t)|^{4}}{|x|}\ dx \lesssim \sup_{t} \|u(t)\|_{\dot{H}^{\frac{1}{2}}}^2.$$
\\
For the monomial defocusing nonlinearity of degree p it follows that the estimate 
$$\int_{\Bbb R_{t}}\int_{\Bbb R^n}\frac{|u(x,t)|^{p+1}}{|x|}dx \lesssim \sup_{t} \|u(t)\|_{\dot{H}^{\frac{1}{2}}}^2$$
is valid for all $n \geq 3$. 
\\
\\
{\it Remark.} Note that this is a nonlinear estimate that we obtained by estimating away the two linear, positive terms. As we have already mentioned in the introduction, in this paper we 
will establish the following estimate that is true for any defocusing nonlinearity in
 two dimensions (this is the reason that we will sometimes call it a linear estimate)
\\
$$\int_{\Bbb R_{t}}\int_{\Bbb R^2}\frac{|u(x,t)|^{4}}{|x|}dx \lesssim \sup_{t} \|u(t)\|_{L^2}^2\|u(t)\|_{\dot{H}^{\frac{1}{2}}}^2.$$ 
\\
The idea goes back to \cite{ckstt4} and we outline the method below. This is a decay estimate, as it shows that the quantity $\int_{\Bbb R^2}\frac{|u(x,t)|^4}{|x|}dx$ must go to zero, at least in some time-average sense. Because
 the weight $\frac{1}{|x|}$ is large at the origin, the solution
 cannot remain lower bounded near the origin for extended period of times. This is a nonlinear effect caused by the defocusing nature of the nonlinearity. 
The estimate is especially useful for spherically symmetric solutions since such solutions already decay away from the origin.
\vskip 0.3 in
We start again with equation \eqref{nls2} in three dimensions and we consider two solutions $u_1(x_{1},t), u_2(x_2,t)$ with $(x_1,x_1) \in \Bbb R^{3} \times \Bbb R^{3}$. The reader can observe that our 
exposition can be carried in any dimension if we work with product functions in $\Bbb R^n \times \Bbb R^n$. We form the tensor product
$$u(x_1,x_2,t)=u_1(x_1,t)u_2(x_2,t):=u_1u_2.$$ 
\\
It is not hard to see that the tensor product satisfies the equation
$$iu_{t}-\Delta_{6} u+f(u)=0$$
where
$$f(u)=|u_1|^2u+|u_2|^2u$$
and $\Delta_{6}$ is the Laplacian in $\Bbb R^6=\Bbb R^3 \times \Bbb R^3$. Thus for $x =(x_1, x_2)\in \Bbb R^6$ we have 
\\
$$\nabla=(\nabla_{x_1},\nabla_{x_2}),\ \ \ \ \ \ \Delta_{6}=\nabla \cdot \nabla =(\nabla_{x_1},\nabla_{x_2}) \cdot (\nabla_{x_1},\nabla_{x_2})= \Delta_{x_1}+\Delta_{x_2}.$$ 
\\
The crucial observation is that the equation 
stays defocusing and we expect to exploit this positivity to obtain new a priori estimates. We write $\rho_{a}=\frac{1}{2}|u_{a}|^2$ for $a=1,2$ and thus
$$\rho=\frac{1}{2}|u|^2=2\rho_{1}\rho_2.$$
In this notation, the local conservation of mass is given by
$$\partial_{t}\rho=\nabla \cdot \Im (\bar{u}\nabla u)$$ 
and the local conservation of momentum is given by
\\
$$\partial_{t}p_{k}=\nabla_{j}\left( \delta_{k}^{j}(\Phi(\rho)-\Delta \rho)+\sigma_{k}^{j} \right)$$
with
$$\Phi(\rho)=4\rho_1\rho_2(\rho_1+\rho_2), \ \ \ \ \ \ p_{k}=\Im (\bar{u}\nabla_k u),\ \ \ \ \sigma_{kj}= 2\Re(\nabla_{j}\bar{u}\nabla_{k}u).$$
\\
Again, contraction with a vector field leads to the equation
\vskip 0.05 in
\begin{equation}\label{vf}
\partial_{t}(X^{k}\ p_{k})=\nabla_{j}\{X^{k}\ \delta_{k}^{j}(\Phi(\rho)-\Delta \rho)+X^{k}\ \sigma_{k}^{j}\}-div \vec{X}(\Phi(\rho)-\Delta \rho)-(\nabla_{j}X^{k})\sigma_{k}^{j}.
\end{equation}
\vskip 0.05 in 
Now we consider $\vec{x}=(\vec{y},\vec{z})\in \Bbb R^6$ and define the distance of $\vec{x}$ from the average of its components $(\frac{\vec{y}+\vec{z}}{2},\frac{\vec{y}+\vec{z}}{2})$. Here
 we assume that $x_1=y$ and that $x_2=z$. Up to a nonessential constant
 this is given by
\\
$$d(\vec{x})=|\vec{y}-\vec{z}|=\sqrt{(y_{1}-z_1)^2+(y_{2}-z_2)^2+(y_{3}-z_3)^2}$$
\\
and the vector field is defined to be the gradient of this distance $X^{j}=\nabla^{j}d$. A quick computation shows that 
$$div \vec{X}=\frac{4}{d}=\frac{4}{|y-z|}$$
and that
$$-\Delta(div \vec{X})=-\Delta_{y}(\frac{4}{d})-\Delta_{z}(\frac{4}{d})=32\pi \delta(\vec{y}-\vec{z}).$$
\\
Thus, if we integrate \eqref{vf} we obtain
\\
$$\partial_{t}M=-\partial_{t}\int_{\Bbb R^{6}}p_{k}X^{k}dx=\int_{\Bbb R^6}(div \vec{X})(\Phi(\rho)-\Delta \rho)dx+\int_{\Bbb R^6}(\nabla_{j} X^{k})\sigma_{k}^{j}dx=$$
\\
$$\int_{\Bbb R^6}(div \vec{X})\Phi(\rho)dx+\int_{\Bbb R^6}-\Delta (div \vec{X})\rho \ dx+\int_{\Bbb R^6}(\nabla_{j} X^{k})\sigma_{k}^{j}dx \geq $$
\\
$$\int_{\Bbb R^6}-\Delta (div \vec{X})\rho \ dx+\int_{\Bbb R^6}(\nabla_{j} X^{k})\sigma_{k}^{j}dx \geq \int_{\Bbb R^6}32\pi \delta (\vec{y}-\vec{z})\frac{1}{2}|u_1(y,t)|^2|u_2(z,t)|^2 \ dx$$
\\
if
$$\int_{\Bbb R^6}(\nabla_{j} X^{k})\sigma_{k}^{j}dx \geq 0.$$
Note that as in the single particle solution case
$$M(t)=-\int_{\Bbb R^{6}}\vec{p}\cdot \vec{X}\ dx=-\int_{\Bbb R^{6}}p_kX^k\ dx.$$
\\
By the fundamental theorem of calculus and after picking $u_1=u_2=u$, we have
\\
$$16\pi \int_{\Bbb R_t}\int_{\Bbb R^3}|u(x,t)|^{4}dx \lesssim \sup_{t}|M(t)|.$$
It remains to estimate $M(t)$. Since $|\vec{X}|\leq 1$ we have that
\\
$$|M(t)| \lesssim \int_{\Bbb R^6}|\vec{p}|\ dx \lesssim \int_{\Bbb R^{3} \times \Bbb R^{3}}\rho(y)|\Im \vec{p}(z)|dy\ dz=\left(\int_{\Bbb R^3}\rho(y)dy\right)\int_{\Bbb R^3}|u(z)||\nabla u(z)|dz.$$
By the momentum estimate in the appendix of \cite{tt},
$$|M(t)| \lesssim \|u(t)\|_{L^2}^2\|u(t)\|_{\dot{H}^{\frac{1}{2}}}^{2}.$$
We obtain
$$\int_{\Bbb R_t}\int_{\Bbb R^3}|u(x,t)|^{4}dx \lesssim \sup_{t}\left( \|u(t)\|_{L^2}^2\|u(t)\|_{\dot{H}^{\frac{1}{2}}}^{2}\right).$$
\\
It remains to show that
$$\int_{\Bbb R^6}(\nabla_{j} X^{k})\sigma_{k}^{j}dx \geq 0.$$
But 
$$X^{j}=\nabla^{j}d=\left\{
\begin{array}{cc}
\frac{y^{j}-z^{j}}{d}, & \text{if} \quad  j=1,2,3\\
\frac{z^{j}-y^{j}}{d}, & \text{if} \quad j=4,5,6.
\end{array}
\right.
$$
\\
The $6\times 6$ matrix $\nabla_{j}X^k$ is given by
$$\nabla_{j}X^k=
\begin{pmatrix}
b_{j}^{k} & -b_{j}^{k}\\
-b_{j}^{k} & b_{j}^{k}
\end{pmatrix}
$$
where $b_{jk}$ is the $3 \times 3$ matrix given by
$$b_{jk}=\frac{\delta_{kj}d^2-(y_k-z_k)(y_j-z_j)}{d^{3}}$$
\\
for $j,k=1,2,3$. Thus, the matrix $\nabla_{j}X^k$ is positive semi-definite and we obtain 
\\
$$\int_{\Bbb R^6}(\nabla_{j} X^{k})\sigma_{k}^{j}dx \geq 0.$$
\\
Since the calculations have appeared in \cite{ckstt4}, we omit the details.
\\
\\
{\it Remark.} We can use the same idea and correlate 4 one dimensional solutions. We define the tensor product of the four solutions and our new vector field is the gradient of the distance function to the
 diagonal $(x,x,x,x) \in \Bbb R^4$. For details see \cite{chvz}. The estimate one can obtain is
\\
$$\int_{\Bbb R_t}\int_{\Bbb R}|u(x,t)|^{8}dx \lesssim \sup_{t}\left( \|u(t)\|_{L^2}^7\|u(t)\|_{\dot{H}^{1}}\right).$$
\\
A variant of this idea will be used in the proof of our main theorem that follows.
\vskip 0.3 in

\begin{proof}[Proof of Theorem \ref{thm1}]
If we repeat the calculations we did for the 3d case to the two dimensional Schr\"odinger equation we obtain
\begin{equation}\label{con2d}
\partial_{t}M=-\partial_{t}\int_{\Bbb R^{4}}p_{k}X^{k}dx=\int_{\Bbb R^4}(div \vec{X})\Phi(\rho)dx+\int_{\Bbb R^4}-\Delta (div \vec{X})\rho \ dx+\int_{\Bbb R^4}(\nabla_{j} X^{k})\sigma_{k}^{j}dx
\end{equation}
where 
\\
$$\rho=\frac{1}{2}|u_1(x_{1},t)|^2|u_2(x_{2},t)|^2=2\rho(x_1)\rho(x_2),\ \ \ \ \Phi(\rho)=4\rho_1\rho_2(\rho_1+\rho_2),$$ 
\\
and $x=(x_1,x_2) \in \Bbb R^2 \times \Bbb R^2$. To find the appropriate vector field consider in $\Bbb R^2$ the line that passes through the origin and
 has direction given by the unit vector $\vec{\omega}$,
$$L(\vec{0},\vec{\omega}):=\{\vec{x}(l)=l\vec{\omega}\}.$$
Without loss of generality we can take $\omega=(1,0)$. We can now lift this line onto a diagonal of $\Bbb R^4$ as follows
$$\tilde{L}(\omega):=\{\vec{x}_1=\vec{x}_2=\vec{x}(l)\}.$$
\\
Consider now the distance of a point 
$$\vec{y}=(\vec{x}_1,\vec{x}_2)\in \Bbb R^2 \times \Bbb R^2:\vec{y}=(y_1,y_2,y_3,y_4) \in \Bbb R^4$$ 
from the line
\\
$$\tilde{L}(\omega):=\{\vec{x}_1=\vec{x}_2=\vec{x}(l)\}=\{\vec{y}\in \Bbb R^4: y_1=l,\ y_2=0,\ y_3=l,\ y_4=0\}.$$
This distance is
\\
$$d=\min_{l \in \Bbb R}\sqrt{(y_1-l)^2+y_2^2+(y_3-l)^2+y_4^2}=$$
\\
$$\min_{l \in \Bbb R}\sqrt{\frac{(y_1-y_3)^2}{2}+y_2^2+y_4^2+2(l-\frac{y_1+y_3}{2})^2}.$$
\\
For $l=\frac{y_1+y_3}{2}$ we have that
$$d=\sqrt{\frac{(y_1-y_3)^2}{2}+y_2^2+y_4^2}.$$
\\
We set $X^{j}=\nabla^{j}d,\ \ j=1,2,3,4$ and compute 
\\
$$X^{1}=\frac{y_1-y_3}{2d},\ \ \ X^{2}=\frac{y_2}{d},\ \ \ X^{3}=\frac{y_3-y_1}{2d},\ \ \ X^{4}=\frac{y_4}{d}.$$
\\
Notice that $|\vec{X}| \leq 1$. We set $s=\frac{y_1-y_3}{\sqrt{2}}$ so that $d^2=\sqrt{s^2+y_2^2+y_4^2}$ and we compute
\\
$$\nabla_1X^1=\frac{1}{2d}-\frac{s^2}{2d^3},\ \ \ \nabla_{2}X^1=-\frac{sy_2}{\sqrt{2}d^3}, \ \ \ \nabla_{3}X^1=-\frac{1}{2d}+\frac{s^2}{2d^3},\ \ \ \nabla_{4}X^1=-\frac{sy_4}{\sqrt{2}d^3}$$
\\
$$\nabla_{1}X^2=-\frac{sy_2}{d^3},\ \ \ \nabla_{2}X^2=\frac{s^2+y_4^{2}}{d^3},\ \ \ \nabla_{3}X^2=\frac{sy_2}{d^3},\ \ \ \nabla_{4}X^2=-\frac{y_2y_4}{d^3}$$
\\
$$\nabla_1X^3=-\frac{1}{2d}+\frac{s^2}{2d^3},\ \ \ \nabla_{2}X^3=\frac{sy_2}{\sqrt{2}d^3}, \ \ \ \nabla_{3}X^3=\frac{1}{2d}-\frac{s^2}{2d^3},\ \ \ \nabla_{4}X^3=\frac{sy_4}{\sqrt{2}d^3}$$
\\
$$\nabla_{1}X^4=-\frac{sy_4}{d^3},\ \ \ \nabla_{2}X^4=-\frac{y_2y_4}{d^3},\ \ \ \nabla_{3}X^4=\frac{sy_4}{d^3},\ \ \ \nabla_{4}X^4=\frac{s^2+y_2^2}{d^3}.$$
\\
We notice that
$$\nabla_{j}X^{k}=\nabla_{k}X^{j}, \ j\ne k,\ \ j,k=1,2,3,4$$
and 
$$div \vec{X}=\nabla_{1}X^1+\nabla_2X^2+\nabla_3X^3+\nabla_4X^4=\frac{2}{d} \geq 0.$$
The first term on the right hand side of \eqref{con2d} is positive and we can write
\\
$$\partial_t M \geq \int_{\Bbb R^4}-\Delta (div \vec{X})\rho \ dx+\int_{\Bbb R^4}(\nabla_{j} X^{k})\sigma_{k}^{j}dx.$$
\\
Now we compute using the fact that the tensor $\sigma_{jk}=2\Re(\nabla_{j}\bar{u}\nabla_{k}u)$ is symmetric
$$\nabla_{k}X^{j}\sigma_{j}^{k}=\sum_{i=1}^{4}\sigma_{i}^{i}\nabla_{i}X^{i}+2\sigma_{12}(\nabla_{2}X^1)+2\sigma_{13}(\nabla_{3}X^{1})+2\sigma_{14}(\nabla_{4}X^1)+$$
\\
$$2\sigma_{23}(\nabla_{3}X^{2})+2\sigma_{24}(\nabla_{4}X^{2})+2\sigma_{34}(\nabla_{4}X^{3}).$$
\\
In addition, we have that
\\
$$\nabla_{1}X^{1}=\nabla_{3}X^3=-\nabla_{3}X^1,\ \ \ \nabla_{2}X^1=-\nabla_{3}X^{2},\ \ \ \nabla_{4}X^1=-\nabla_{4}X^3$$
and thus
\\
$$\nabla_{k}X^{j}\sigma_{j}^{k}=(\sigma_{11}+\sigma_{33}-2\sigma_{13})\nabla_{1}X^{1}+2(\sigma_{12}-\sigma_{23})\nabla_{1}X^2+2(\sigma_{14}-\sigma_{34})\nabla_{1}X^4+$$
\\
$$2\sigma_{24}\nabla_{2}X^4+\sigma_{22}\nabla_{2}X^2+\sigma_{44}\nabla_{4}X^4.$$
\\
For $s=\frac{y_1-y_3}{\sqrt{2}}$ we have $\nabla_{s}=\frac{\nabla_{y_{1}}-\nabla_{y_3}}{\sqrt{2}}$
and thus
\\
$$\sigma_{11}+\sigma_{33}-2\sigma_{13}=2|(\nabla_{1}-\nabla_{3})u|^2=4|\nabla_{s}u|^2,$$
\\
$$2(\sigma_{12}-\sigma_{23})=4\Re(\nabla_{2}u\nabla_{s}\bar{u}),\ \ \ \ 2(\sigma_{14}-\sigma_{34})=4\Re(\nabla_{4}u\nabla_{s}\bar{u}),\ \ \ \ 2\sigma_{24}=4\Re(\nabla_{4}u\nabla_{4}\bar{u}).$$
\\
If we put everything together we compute
\\
$$\nabla_{k}X^{j}\sigma_{j}^{k}=4|\nabla_{s}u|^2\left(\frac{1}{2d}-\frac{s^2}{2d^3}\right)-4\frac{sy_2}{d^3}\Re(\nabla_{2}u\nabla_{s}\bar{u})-4\frac{sy_4}{d^3}\Re(\nabla_{4}u\nabla_{s}\bar{u})-$$
\\
$$4\frac{y_2y_4}{d^3}\Re(\nabla_{2}u\nabla_{4}\bar{u})+2|\nabla_{2}u|^{2}\frac{s^2+y_4^2}{d^3}+2|\nabla_{4}u|^{2}\frac{s^2+y_2^2}{d^3}=$$
\\
$$\frac{2}{d}\left( |\nabla_{s}u|^2+|\nabla_{2}u|^2+|\nabla_{4}u|^2-\frac{1}{d^2}|(s\nabla_{s}+y_{2}\nabla_{2}+y_{4}\nabla_{4})u|^2\right).$$
\\
Thus if $z=(s,y_2,y_4)$ we obtain
\\
$$\nabla_{k}X^{j}\sigma_{j}^{k}=\frac{2}{d}\left( |\nabla_{z}u|^2-\frac{1}{|z|^2}|(z\cdot \nabla_{z})u|^2\right) \geq 0.$$
\\
Since this term is also positive we have
\\
$$\partial_t M \geq \int_{\Bbb R^4}-\Delta (div \vec{X})\rho \ dx.$$
But 
$$\Delta=\frac{\partial^2}{\partial y_{1}^{2}}+\frac{\partial^2}{\partial y_{2}^{2}}+\frac{\partial^2}{\partial y_{3}^{2}}+\frac{\partial^2}{\partial y_{4}^{2}}=\frac{\partial^2}{\partial s^{2}}+
\frac{\partial^2}{\partial y_{2}^{2}}+\frac{\partial^2}{\partial y_{4}^{2}}=\Delta_{(s,y_2,y_4)}.$$
\\
Then
$$-\Delta (div \vec{X})=-\Delta (\frac{2}{d})=-\Delta_{z}(\frac{2}{|z|})=8\pi \delta(\vec{z})=8\pi \delta(\vec{x_1}=\vec{x_2}=\vec{x}(l))$$
and
$$\int_{\Bbb R^4}-\Delta (div \vec{X})\rho \ dx=\int_{\Bbb R^4}8\pi \delta(\vec{x_1}=\vec{x_2}=\vec{x}(l))\rho\ dx=8\pi\int_{L}\rho(\vec{x}(l),t)dl.$$
Thus
$$\int_{\Bbb R^4}-\Delta (div \vec{X})\rho \ dx=2\pi\int_{L}|u(\vec{x}(l),t)|^{4}dl,$$
\\
if we pick
$$\rho_{1}=\frac{1}{2}|u_1|^2=\frac{1}{2}|u_2|^2=\rho_{2}=\frac{1}{2}|u|^2$$
in which case 
$$\rho =2\rho_1 \rho_2=\frac{1}{2}|u|^4.$$
We now apply the fundamental theorem of calculus and obtain
\\
$$\int_{\Bbb R}\int_{L}|u(\vec{x}(l),t)|^{4}dldt \lesssim \sup_{t}|M(t)|,$$ 
\\
and estimating as before we have
\\
$$\int_{\Bbb R}\int_{L}|u(\vec{x}(l),t)|^{4}dldt \lesssim  \sup_{t}\|u(t)\|_{L^2}^2\|u(t)\|_{\dot{H}^{\frac{1}{2}}}^{2}.$$
\\
Now consider the line $L_{i}$ that passes through the origin and has an angle $\theta_{i}$ with the $x-$axis. We parametrize the line by $\vec{x}_{i}(r)=r\vec{\omega}_{i}$ where $\omega_{i}=e^{i\theta_{i}}$. 
We thus identify $\omega_{i}=(\cos(\theta_{i}), \sin(\theta_{i}))$ where $\theta_{i}$ is fixed.
We have proved that
\\
$$\int_{\Bbb R}\int_{L_{i}}|u(\vec{x}(l),t)|^{4}dldt =\int_{\Bbb R}\int_{0}^{\infty} |u(re^{i\theta_{i}},t)|^{4}drdt \lesssim  \sup_{t}\|u(t)\|_{L^2}^2\|u(t)\|_{\dot{H}^{\frac{1}{2}}}^{2}.$$
\\
We cut the unit disk into $N$ sectors each one confined within $\Delta \theta_{i}=\frac{2\pi}{N}$. We apply the previous inequality for every direction $\vec{\omega}_{k}=e^{ik\frac{2\pi}{N}}$ for $k=1,2,...,N$. 
We then divide by $N$. This averaging argument yields
\\
$$\frac{1}{N}\sum_{i=1}^{N}\int_{\Bbb R}\int_{0}^{\infty} |u(re^{i\theta_{i}},t)|^{4}drdt \lesssim  \sup_{t}\|u(t)\|_{L^2}^2\|u(t)\|_{\dot{H}^{\frac{1}{2}}}^{2}$$
\\
or
$$\int_{\Bbb R}\int_{0}^{\infty} \sum_{i=1}^{N}|u(re^{i\theta_{i}},t)|^{4}\frac{dr}{N}dt\lesssim  \sup_{t}\|u(t)\|_{L^2}^2\|u(t)\|_{\dot{H}^{\frac{1}{2}}}^{2}$$
\\
or
$$\int_{\Bbb R}\int_{0}^{\infty} \frac{1}{2\pi}\sum_{i=1}^{N}|u(re^{i\theta_{i}},t)|^{4}\Delta \theta_{i}drdt\lesssim  \sup_{t}\|u(t)\|_{L^2}^2\|u(t)\|_{\dot{H}^{\frac{1}{2}}}^{2}.$$
\\
Taking $N\rightarrow \infty$ we have
\\
$$\int_{\Bbb R}\int_{0}^{\infty}\int_{0}^{2\pi}|u(re^{i\theta},t)|^{4}d\theta drdt\lesssim  \sup_{t}\|u(t)\|_{L^2}^2\|u(t)\|_{\dot{H}^{\frac{1}{2}}}^{2}$$
and
\\
$$\int_{\Bbb R}\int_{0}^{\infty}\int_{0}^{2\pi}\frac{|u(re^{i\theta},t)|^{4}}{r}d\theta (rdr)dt= \int_{\Bbb R_{t}\times \Bbb R^2} \frac{|u(x,t)|^{4}}{|x|}dxdt \lesssim  \sup_{t}\|u(t)\|_{L^2}^2\|u(t)\|_{\dot{H}^{\frac{1}{2}}}^{2}.$$
\\
Since the Schr\"odinger equation is translation invariant for any $x_{0} \in \Bbb R^2$ we have that
\\
$$\sup_{x_{0}}\int_{\Bbb R_{t}\times \Bbb R^2} \frac{|u(x,t)|^{4}}{|x-x_{0}|}dxdt \lesssim  \sup_{t}\|u(t)\|_{L^2}^2\|u(t)\|_{\dot{H}^{\frac{1}{2}}}^{2}$$
\\
and the theorem is proved.
\end{proof}
{\it Remarks.} a) One can consider any curve on $\Bbb R^2$ and lift it onto a diagonal of $\Bbb R^4$ as follows
$$\tilde{L}(\omega)=\{\vec{x}_1=\vec{x}_2=\vec{x}(l)\},$$
\\
where $\vec{x}(l)$ is the parametric representation of the curve. Following a similar strategy we can prove the following a priori estimate on the restricted $L_{t}^4L_{C}^{4}$ norm of $u$
\\
$$\int_{\Bbb R}\int_{C}|u(\vec{x}(l),t)|^{4}dldt \lesssim  \sup_{t}\|u(t)\|_{L^2}^2\|u(t)\|_{\dot{H}^{\frac{1}{2}}}^{2}.$$
\\
b) The method of contracting the momentum equation with vector fields
that are the gradient of certain distant functions is quite general. One can obtain estimates in higher dimensions using this method. We will
 discuss these estimates and their application in future work.

\end{document}